\documentclass[11pt]{article}
\usepackage{amsmath,amssymb,amsthm,comment,amsfonts,mathrsfs}
\usepackage{amsfonts}
\usepackage[latin1]{inputenc}
\usepackage{color}
\usepackage{mathtools}
\usepackage{hyperref}

\usepackage[top=1in,bottom=1in,left=1in,right=1in]{geometry}

\newtheorem{lem}{Lemma}[section]

\newtheorem{thm}{Theorem}[section]
\newtheorem{lemma}[thm]{Lemma}
\newtheorem{cor}[thm]{Corollary}
\newtheorem{remark}[thm]{Remark}
\newtheorem{prop}[thm]{Proposition}

\numberwithin{equation}{section}
\newtheorem{example}{Example}[section]

\newcommand{\E}{\mathcal{E}}

\newcommand{\eps}{\varepsilon}
\newcommand{\R}{\mathbb{R}}
\newcommand{\N}{\mathbb{N}}

\newcommand{\lap}{\Delta}

\newcommand{\be}{\begin{equation} \label}
\newcommand{\ee}{\end{equation}}

\begin{document}

\begin{center}		
	\section*{Refined blow-up behavior for semilinear heat equations with arbitrarily fast growing nonlinearities}
	
\end{center}

\begin{center}
	Loth Damagui CHABI \footnote{Email: chabi@math.univ-paris13.fr} \\
		Universit\'e Sorbonne Paris Nord \& CNRS UMR 7539, Laboratoire Analyse G\'eom\'etrie et Applications, 93430 Villetaneuse,  France.
	\\
		$ $
	\end{center}

\begin{abstract}
We study radially nonincreasing, positive blow-up solutions of a class of superlinear heat equations in the non-scale-invariant setting.
This class contains nonlinearities with arbitrarily fast, super-exponential growth.
 Under suitable structural assumptions on the nonlinearity, we derive the sharp global blow-up profile in the original space-time variables.
 More precisely, for positive nonlinearities $f$ satisfying 
$$
 \lim_{X\to \infty}f'(X)G(X)= q\ge1\quad \quad 
 \hbox{  where }\quad\quad G(X)=\int_{X}^{\infty}\frac{ ds}{f(s)}<\infty\quad\forall\ X>0,
$$ 
and assuming either the time monotonicity $u_t\ge0$ holds or $q>\frac{p_S}{p_S-1}$, we have
	 $$
	u(x,t)=(1+o(1))\,G^{-1}\bigg(T-t+\frac{|x|^2}{4q|\log |x|^2|}\bigg),
		\quad \hbox{as $(x,t)\to (0,T)$.} 
$$
This estimate in particular provides the sharp final space profile and the corresponding refined space-time profile. Moreover, our result recovers the universal structure of the global blow-up profile as in \cite{C25-2,CS25}, given by the resolvent $G^{-1}$ of the associated ODE evaluated at a time-space building block, which only depends on the limit $q$.

\vskip 3pt

\noindent{\bf Key words:} Semilinear heat equation, blow-up profile, refined space-time behavior, super-exponential nonlinearities, regularly varying functions.

\vskip 3pt

\noindent{\bf AMS subject classifications:} 35K58, 35B44, 35B40, 35K05

\end{abstract}

\section{Introduction}

Consider the following semilinear heat equation 
	\begin{equation}\label{SHE}
	\begin{cases}
	u_t-\Delta u=f(u),\quad x\in \Omega,\ t>0,\\
	u=0,\quad x\in \partial\Omega,\ t>0,\\
	u(x,0)=u_0(x),\quad x \in \Omega,
	\end{cases},
	\end{equation}
	where $\Omega \subset \mathbb{R}^n\ (n\ge 1)$ is smooth, $u_0\in L^\infty(\Omega)$  and $f\in C^1([0,\infty))$.
For the model cases 
$$\hbox{$f(u)=|u|^{p-1}u$ \quad ($p>1$) \quad and \quad $f(u)=e^u$},$$
the blowup behavior has been the subject of extensive research
	and is rather well understood, especially in the Sobolev subcritical power case $1<p<(n+2)/(n-2)_+$
	 where a complete classification of blow-up profiles is available.
	For $f(u)=|u|^{p-1}u$ we refer to \cite{bebernes1987description,BK94,filippas1992refined,GK85,GK87,
GK89,GMS04,MM2004a,MM2004b,merle1998optimal,merle1998refined,PQS07b,Q21,quittner2019superlinear,
S19,Vel93b,V92}
	 and, for $f(u)=e^u$, to \cite{bebernes1987description,bebernes1992final,BT, bressan1990asymptotic,E,ET,FP08,
FM85,HV93,pulkkinen2011blow,S22,W}.
	However the analysis of sharp blow-up asymptotics in these works heavily depends on the exact form of the nonlinearities and their scale invariance properties, 
	namely the invariance of the equations under the transformations 
	$$u_\lambda(x,t)=\lambda^{2/(p-1)}u(\lambda x, \lambda^2 t) \quad (\lambda >0),$$
and
$$u_\lambda(x,t)=u(\lambda x, \lambda^2 t)+2\log \lambda \quad (\lambda >0)$$
respectively (or at least on an asymptotic homogeneity of the nonlinearities as the variable $u\to\infty$).

This paper is a contribution to the understanding of the precise singularity formation 
	for solutions of superlinear parabolic problems involving {\it non scale invariant and arbitrarily fast growing nonlinearities},
in continuation to the works \cite{duong2018construction,HZ22,souplet2022universal,CS25,QSLiouv,FK25,C25-1,C25-2,Ryoto} (see below for more details on these works).

	\smallskip

Let us briefly review the known results on the blow-up behavior of solutions to problem \eqref{SHE} for non-scale-invariant nonlinearities.
In the early work \cite{FM85}, blow-up rate estimates were established for positive, time-increasing solutions of \eqref{SHE} with rather general nonlinearities. For suitable classes of non-scale-invariant nonlinearities, results concerning blow-up sets were obtained in \cite{F14}. In \cite{duong2018construction,HZ22}, for problem \eqref{SHE} with the particular logarithmically perturbed nonlinearity
$f(u)=|u|^{p-1}u\log(2+u^2)$,
blow-up rate estimates were derived without sign or monotonicity assumptions, and special solutions with a prescribed blow-up profile were constructed (see also \cite{hamza2021,RZ} for blow-up rate estimates for the related nonlinear wave equations).
In \cite{C25-1,CS25,QSLiouv,souplet2022universal}, positive solutions of problem \eqref{SHE} were studied for general regularly varying nonlinearities of the form
$f(u)=u^pL(u)$,
where $L$ is slowly varying function. In \cite{QSLiouv,souplet2022universal}, parabolic Liouville-type theorems were established, together with applications to universal bounds and blow-up rate estimates. An analogue of the classical Giga--Kohn result was proved in \cite{C25-1}, using, among other ingredients, the type~I blow-up estimate from \cite{souplet2022universal}. Moreover, building on the results of \cite{C25-1}, \cite{CS25} established the final and space-time blow-up profiles of positive, radially decreasing solutions, revealing a {\it structural universality} of the blow-up profile that is robust with respect to the slowly varying factor $L(u)$.
The paper \cite{C25-2} considers nonlinearities of the form
$f(u)=e^uL(e^u)$.
The author obtained a complete description of the blow-up profile for radially symmetric, nonincreasing solutions under the assumption of time monotonicity. This class includes the nonlinearity $f(u)=u^pe^u$. Then the author of \cite{Ryoto} investigated the more general super-exponential nonlinearities
$f(u)=u^qe^{u^p}$,
where $p>1$ and $q\in\{0\}\cup[1,\infty)$. An analogue of the Giga--Kohn property was established in \cite{Ryoto}, but only in dimensions $n=1,2$. More recently, a type I blow-up rate result was obtained for more general nonlinearities in radial setting \cite{FK25}.

	\smallskip

In this paper, we are interested in non-scale-invariant nonlinearities with arbitrarly fast growth at infinity. We note that, as far as blow-up rate estimates are concerned, the results in \cite{FM85} and the recent works \cite{FK25} and \cite[Theorem~3.1]{CS25} already cover such nonlinearities under suitable assumptions.
Our main goal is to provide a precise description of the blow-up asymptotics, including the space-time and final blow-up profiles, for nonnegative radially decreasing solutions of problem \eqref{SHE}. Our analysis builds in part on the approach developed in our previous works \cite{C25-1,CS25} for regularly varying nonlinearities of power type, as well as in \cite{C25-2} for nonlinearities of exponential type.
   
   \section{Main results}

Throughout the paper, the nonlinear term $f$ is assumed to satisfy
	 \begin{equation}\label{sem0}
	 f\in C^1([0,\infty)),\quad f(0)\ge 0,\ \hbox{  $f$ is $C^2$ and positive for large $s$}
	 \end{equation}
	 \begin{equation}\label{defF}
	f(s)\to \infty\quad\hbox{as } s\to \infty\quad \hbox{and }\  G(X):=\int_{X}^{\infty}\frac{ds}{f(s)}<\infty\quad\hbox{for any }\ X>0.  
	 \end{equation} 
	By standard theory, 
	 for $u_0\in L^\infty(\Omega)$ with $u_0\ge 0$ (suitably large), problem \eqref{SHE} admits a unique nonnegative classical solution with maximal existence time $T = T(u_0) \in (0,\infty),$ which we denote by $u$ throughout this paper. Moreover, $u$ blows up in the sense that
	\begin{equation*}
	\lim_{t\to T} \|u(t)\|_\infty=\infty.
	\end{equation*}
	 In this case, $T$ is called the blowup time of $u$. Given $a \in \overline{\Omega}\subset\mathbb{R}^n$, we say that $a$ is a blowup point of $u$ if there exists $(a_j,t_j)\to(a,T)$ such that $u(a_j,t_j)\to\infty$ as $j\to \infty$. 
	 We also note that there exists a large $X_1>0$ such that $G:[X_1,\infty)\to (0,G(X_1)]$ is decreasing and that
 $$	 
 \psi(t):=G^{-1}(T-t)
$$
 is the positive solution of the corresponding ODE, i.e $y'(t)=f(y)$, which blows up  at the same time $t=T$.

We describe conditions to be assumed in our main results.
For the initial data, we assume
			\begin{equation}\label{Omega-u0}
		\hbox{$\Omega=\mathbb{R}^n$ , $0\le u_0\in L^\infty(\Omega)$, radially symmetric, nonincreasing in $r = |x|$ and nonconstant.}			 
			 \end{equation} 
This guarantees that $u(x,t)=u(r,t)\ge 0$ with $r=|x|$ and $u_r\le 0$. Moreover, by a simple comparison principle, we have 
\begin{equation}\label{belowrate}
	u(0,t)=\|u(t)\|_{L^\infty(\Omega)} \geq G^{-1}(T-t)\quad\  \text{for}\ t \in (0,T) .
\end{equation}

Concerning the nonlinear term $f$, in addition to \eqref{sem0}-\eqref{defF}, we impose that there exist $q\ge 1$  such that
 \be{qlim}
 \Big(\frac{f}{\log f}\Big)''\ge0 \hbox{ and }\ f'(X)G(X)=q+o(|\log G(X)|^{-\alpha_0})\hbox{ for large $X$ and some $\alpha_0\in(1/2,1)$.}
 \ee
This condition for $q=1$ allows nonlinearities with arbitrarily fast superexponential growth (examples will be given below).

Our first main result is the following global, refined blow-up estimate, valid in the scale of the original variables~$(x, t)$.

\begin{thm}\label{Thm} 
 Let $f$ satisfy the assumptions \eqref{sem0}-\eqref{defF} and \eqref{qlim}. Let $u_0$ and $\Omega$ be as in \eqref{Omega-u0} such that $T=T(u_0)<\infty$. Assume that  $u_t\ge 0$ or $q>q_S:=p_S/(p_S-1)$. Then we have
	\begin{equation}\label{infer1}
	u(x,t)=(1+o(1))G^{-1}\bigg(T-t+\frac{|x|^2}{4q|\log|x|^2|}\bigg)\quad \text{as } (x,t)\to (0,T).
	\end{equation}
\end{thm}

On the other hand, since $u\geq0$ is radially symmetric decreasing in $|x|$ for all $t\in (0,T)$ and $0$ is the only blow-up point of $u$,
 standard parabolic estimates yields that $u(x,T):=\lim_{t\to T}u(x,t)$ exists for all $x\ne 0$.
Hence, as a direct consequence of Theorem~\ref{Thm}, 
we obtain a sharp description of the final blow-up profile and the refined space-time profile:

\begin{cor}\label{Debo}
	Under the assumptions of Theorem~\ref{Thm}, we have 
	
		\begin{itemize}
			\item [(i)]\textbf{(final profile)} \begin{equation}\label{fp}
			u(x,T)=(1+o(1)) G^{-1}\bigg(\frac{|x|^2}{4q|\log|x|^2|}\bigg)   \quad\text{as }\ x\to 0;
			\end{equation}
			\item [(ii)] \textbf{(refined space-time profile)}
			\begin{equation}\label{lou}
			u\bigl(\xi\sqrt{(T-t)|\log(T-t)|},t\bigr)= (1+o(1))G^{-1}\bigg[(T-t)\bigg(1+\frac{|\xi|^2}{4q}\bigg)\bigg] 
			\quad\text{as }\ t\to T,
			\end{equation}
			uniformly for $\xi$ bounded.
		\end{itemize}
\end{cor}

As part of the proof of Theorem \ref{Thm}, we derive an analogue of Giga-Kohn property with key second term, for bounded $y$. This expansion is of independent interest because it yields a sharper asymptotic description of the solution than \eqref{infer1}, although it is valid only in the more restrictive regime 
$|x|=O(\sqrt{T-t})$. 
 Here and in the rest of this paper, we denote 
	\begin{equation}\label{defL2H1}
	L^2_\rho:=\Bigl\{v\in L^2_{loc}(\mathbb{R}^n);\ \int_{\mathbb{R}^n}v^2(y)\rho(y)dy<\infty\Bigr\},\quad 
	H^1_\rho:=\bigl\{f\in L^2_\rho;\ \nabla f\in L^2_\rho\bigr\},\quad \rho(y):=e^{-|y|^2/4}.
	\end{equation}

\begin{thm}\label{theo}
	Under assumptions of Theorem \ref{Thm}, we have:

$(i)$ if $q=1$
\begin{equation}\label{A1}
\log \frac{T-t}{G(u(y\sqrt{T-t},t))}=-\frac{ |y|^2-2n+o(1)}{4|\log(T-t)|},\quad \text{as} \ t\to T,
	\end{equation}
	
$(ii)$ if $q>1$
\begin{equation}\label{A1-2}
\Big(\frac{T-t}{G(u(y\sqrt{T-t},t)}\Big)^{(q-1)}=1-\frac{q-1}{q}\frac{ |y|^2-2n+o(1)}{4|\log(T-t)|},\quad \text{as} \ t\to T,
	\end{equation}
	with convergence in $H^1_\rho(\mathbb{R}^n)$ and uniformly for bounded $y$.
\end{thm}

\section{Discussion and  outline of proofs}

\subsection{ Discussion}\label{L11}

\begin{remark} 
$(i)$ {\bf(Alternative type I assumption)}
For $n\le 2$, Theorems \ref{Thm}-\ref{theo} remain valid 
under the type I blow-up assumption, namely,
 \begin{equation}\label{dam1a}
	\|u(\cdot,t)\|_\infty\le G^{-1}(C(T-t)),\quad T-\delta<t<T,
	\end{equation}
for some $C,\delta>0$.
In any dimension, property \eqref{dam1a} is in particular guaranteed if $u_t\ge 0$
(see \cite[Lemma~4.7]{C25-2}, which follows from arguments in \cite[Section 4]{FM85}).
However, in higher dimensions, the conclusions of Theorems \ref{Thm}-\ref{theo}
are not true in general under the weaker assumption \eqref{dam1a}.
In \cite{souplet2022universal}, it proved that \eqref{dam1a} holds for $q>p_S/(p_S-1)$ for any positive solution of \eqref{SHE}. This was extended in \cite{FK25} to the case of $q=1$ and $3\le n\le 9$ or $q_{JL}<q<q_S$ and $n\ge3$  in the radial setting, where 
$$
q_{JL}=\begin{cases} 1\quad\hbox{if }\ n\le 10,\\
\frac{n-2\sqrt{n-1}}{4}\quad\hbox{if } n\ge 11.
\end{cases}
$$

$(ii)$ {\bf(Bounded domain)} For simplicity, we have restricted our study to the case $\Omega=\mathbb{R}^n$. However, the same arguments apply, with minor modifications, to the case $\Omega=B_R$, following the approach in \cite{C25-2,CS25}.

\end{remark}

\begin{remark} \label{rem1}
 {\bf Structural universality of the profile.}
 As in \cite{C25-2,CS25}, our main result recovers  
 the universal structure of the global blow-up profile. More precisely, the profile is obtained by applying the ``resolvent'' $G^{-1}$ of the associated ODE to the universal time-space building block
\begin{equation}\label{block}
T-t+\frac{|x|^2}{4q|\log|x|^2|}.
\end{equation}
It turns out that
 the influence of $f$ on the building block appears only at the coefficient of the logarithmic spatial correction which is governed by 
 \be{checkq}
 \lim_{X\to\infty}\Big(f'(X)\int_X^\infty\frac{ds}{f(s)}\Big)=q.
 \ee
  The same phenomenon also holds for the final and refined space-time profiles, as well as for the sharp asymptotics in backward space-time parabolas (cf.~\eqref{fp}, \eqref{lou} and \eqref{A1}).
\end{remark}

\begin{example}{\bf(Examples of nonlinearities)}\label{exampleofL}
Our results apply in particular to the fast growing nonlinearities $f(s)=exp_m(u^p)$, with $m,p\ge 1$,
where $\exp_0:=id$ and $\exp_m(\cdot)=\exp(\exp_{m-1}(\cdot))$, and they are already completely new in this case. More generally, typical examples of nonlinearities fulfilling assumptions \eqref{sem0}-\eqref{qlim},
so that Theorems~\ref{Thm}-\ref{theo} apply,
 are given by~ 
\be{examplesf} 
 f(u)=\exp_m(u^p)L(\exp_m(u^p))\quad\hbox{  for }\ m\in\{0,1,2,\cdots\}\quad\hbox{ and }\  \begin{cases}p>1\quad \hbox{if  $m=0$},\\
 p\ge1\quad \hbox{if  $m\in \N$}\end{cases}
\ee 
  where $L\in C^1([0,\infty))$ is a slow variation function satisfying 
 \be{semC-CS}
 \frac{XL'(X)}{L(X)}=o(\log^{-\alpha_0} X)\quad\hbox{ and }\quad \frac{X^2L''(X)}{L(X)}=o(1)\quad\hbox{as }\ X\to \infty.
\ee
Such $L$ include logarithms and their powers and iterates, as well as some strongly oscillating functions. 
  We here provide a few examples of such function $L$:
$$
 	\left\{\begin{aligned}
 		&\hbox{\ $\star$\  $\log^a (K+s)$ for $K>1$ and $a\in\R$ or $K=1$ and $a\ge1$,} \\
 		\noalign{\vskip 1mm}
 		&\hbox{\ $\star$\  the iterated logarithms $\log_m(K+s)$, with $K=\exp\circ\exp\circ\cdots\circ \exp(0)$,}  \\
 		\noalign{\vskip 1mm}
 		&\hbox{\ $\star$\  $\exp(|\log s|^\nu)$ with $\nu\in (0,1/2)$,} \\
 		\noalign{\vskip 1mm}
 		&\hbox{\ $\star$\  the strongly oscillating functions} \\
 		&\hbox{\qquad $\bigl[\log(3+s)\bigr]^{\sin[\log\log(3+s)]}
 			\quad\hbox{ and }\quad
 			\exp\bigl[|\log s|^\nu\cos(|\log s|^\gamma)\bigr],\ \ \nu, \gamma>0,\ \nu+\gamma<1/2$,} \\
 		\noalign{\vskip 1mm}
 		&\hbox{\ $\star$\  $1 +a \sin\bigl(\log^\nu(2+s)\bigr)$ with $ \nu\in (0,1/2)$ and $|a|<1$.}
 	\end{aligned}
 	\right.
 	$$
 	We observe that the factor $L$ does not affect the limit $q$ in \eqref{qlim}. 
Indeed, a sufficient condition for the limit $q$ in \eqref{checkq} to exist is that
$$
\lim_{X\infty}\frac{f'^2}{ff''}=q.
$$
\end{example}
\begin{remark}{\bf (Known results on the blow-up profile)}
$(i)$ Regarding \eqref{examplesf}, we see that 
$$
\begin{cases}
q=1\quad\hbox{if }\quad m\in\N \quad \hbox{and }\ p\ge1\\
q=\frac{p}{p-1}\quad\hbox{if }\quad m=0 \quad \hbox{and }\ p>1.\\
\end{cases}
$$

Our results cover all the nonlinearities considered in \cite{C25-2} and \cite{CS25}, namely $e^uL(e^u)$ and $u^pL(u)$ with $p>1$, respectively. We obtain the same results under weaker assumptions on $f$. Indeed, they obtained the sharp global blow-up profile given in \eqref{infer1} under the  assumptions \eqref{semC-CS} and 
$$
\Big(\frac{XL'(X)}{L(X)}\Big)'=o\Big(\frac{1}{X\log X}\Big) \quad \hbox{as }\quad X\to\infty.
$$
This assumption is used in their proofs only to control the linear term $\tau(s)$ in the rescaled equation; see \cite[Formula~(6.14)]{C25-2} and \cite[Formula~(6.27)]{CS25}. This term disappears in our rescaled equation, thanks to the transformation \eqref{dam11} below. Consequently, we are able to obtain the same sharp blow-up profile without requiring this additional assumption involving $L''$.

$(ii)$ Choosing $m=1$ and $p>1$ in \eqref{examplesf}, $a=q/p$ and $L(s)=\log^a(K+s)$, with either $K>0$ and $a=0$, or $K=0$ and $a\ge1/p$, we obtain
\be{ryoequ}
f(u)=u^q e^{u^p},
\ee
where $q\in\{0\}\cup[1,\infty)$. Thus, our assumptions cover the entire class of nonlinearities considered in \cite{Ryoto}. In that case, we have 
$$
G^{-1}(Y)=(1+o(1))|\log Y|^{1/p}\quad\hbox{as}\quad Y\to 0^+.
$$
Combining this with \eqref{infer1}, we obtain the explicit blow-up asymptotics of $u$, which takes the form
$$
u(x,t)= (1+o(1))\Big|\log\bigg(T-t+\frac{|x|^2}{4|\log|x|^2|}\bigg)\Big|^{1/p}\quad \text{as } (x,t)\to (0,T).
$$
\end{remark}

\begin{remark}\label{CSLem} {\bf (On the convexity assumption)}
Let $f$ satisfy \eqref{sem0}-\eqref{defF} and that 
$$
\lim_{X\to \infty} f'(X)G(X)=q>1.
$$
Then $f$ has regular variation at $\infty$ with index $p=\frac{q}{q-1}$. This follows by showing that $L(X)=X^{-p}f(X)$ satisfies
$$
\lim_{X\to \infty} \frac{XL'(X)}{L(X)}=0.
$$
 Moreover, if $f$ is a $ C^2$ function at $\infty$, it was proved in \cite[Page~4526]{CS25} that $\frac{f}{\log f}$ is convex for large $X$, assuming
$$
\lim_{X\to \infty} \frac{X^2L''(X)}{L(X)}=0.
$$

\end{remark}

\subsection{Outline of proofs}

A first main ingredient of the proof is the following analogue of the Giga-Kohn property, whose proof follows the argument of \cite[Proposition~3.3]{C25-2} and \cite[Theorem~2.1]{C25-1}.

\begin{prop}\label{qssblowup}
Let $p > 1$.
 Assume \eqref{sem0}-\eqref{defF}, \eqref{qlim}, \eqref{Omega-u0} and $T=T(u_0)<\infty$. We also assume that $u$ is of type~I. For any $t_j\to T$, there exist a subsequence (still denoted $t_j$) and a nonnegative radially decreasing function $\phi$ solving
$$
\Delta\phi-\frac12 y\cdot\nabla\phi=\begin{cases}1-e^\phi\quad\hbox{if }\ q=1,\\ (q-1)\phi-\phi^{q/(q-1)}\quad\hbox{if }\ q>1,\end{cases}
\quad\text{in }\mathbb{R}^n,
$$
 such that
\begin{equation}\label{eq:qssblowup}
	\phi(|y|)=\lim_{t_j \rightarrow T} \begin{cases}\log\Big(\frac{(T-t_j)}{G(u(y\sqrt{T-t_j},t_j))}\Big)\quad\hbox{if }\ q=1,\\ \Big(\frac{(q-1)(T-t_j)}{G(u(y\sqrt{T-t_j},t_j))}\Big)^{q-1} \quad\hbox{if }\ q>1,\end{cases}
\end{equation}
uniformly for $y$ bounded. In particular, under the assumptions of Theorem~\ref{Thm}, \eqref{eq:qssblowup} holds with $\phi\equiv(q-1)^{q-1}$ and $\phi\equiv0$ if $q=1$.
\end{prop}

Note that \eqref{eq:qssblowup} is analogous to \cite[Theorem~1.1]{Ryoto}, where \eqref{ryoequ} plays the role of the nonlinearity, and the result is established in dimensions $n\le 2$ with $\phi\equiv0$. For $f(s)=e^s$, corresponding to $m=1,\ p=1$ and $L\equiv1$ in \eqref{examplesf}, Proposition~\ref{qssblowup} was obtained in \cite{BE}.
The proof in \cite{BE} is based on similarity variables and energy arguments inspired by \cite{GK85}, combined with an intersection-comparison analysis
involving the singular steady-state $U^*$ of $\Delta u+e^u=0$.

The key point in the proof of Proposition~\ref{qssblowup} and Theorem~\ref{theo} lies in the following change of variables. Indeed, setting
$$
y=\frac{x}{\sqrt{T-t}}\quad\hbox{and }\ s=-\log(T-t),
$$
and introducing the new function
$$
w(y,s)=\Psi\Big(\frac{T-t}{G(u(x,t))}\Big),\quad\hbox{where }\ \Psi(X)=\begin{cases} \log X\quad \hbox{if }q=1,\\ ((q-1)X)^{q-1}\quad \hbox{if }q>1,\end{cases}
$$
the PDE in \eqref{SHE} becomes 
$$
w_s-\Delta w +\frac{1}{2}y\cdot\nabla w=f_1(w)+ H_0(s,y),
$$
where $f_1$ and $H_0$ are defined in \eqref{VraieEE}-\eqref{VraieEEmod} below.
We note that this change of function is inspired by \cite{FI18} and was used in \cite{FI22,FI23,M18} in a different way.

To prove Proposition~\ref{qssblowup}, following the approach in \cite{C25-1,C25-2} based on energy functional, it suffices to have 
$$
\|H_0(s,\cdot)\|_\infty\le Cs^{-\alpha}\quad \hbox{for some }\ \alpha>\frac12.
$$
It turns out that the latter estimate is guaranteed by our assumption \eqref{qlim} together with the no-needle property established in Lemma~\ref{noneedle}.

Building on Proposition~\ref{qssblowup}, the proof of Theorem~\ref{theo} relies on a simplified center-manifold-type analysis similar to that developed in \cite[Section~6]{CS25}. Namely, instead of employing the ODE renormalization as in \cite{CS25,C25-2}, we first rescale the solution using the change of variables introduced above. We then linearize around the constant solution $\kappa=(q-1)^{q-1}$, with $\kappa=0$ when $q=1$. Our analysis shows that the perturbation $H_0$ does not affect the control of the resulting linearized equation. Consequently, the arguments in \cite[Section~6]{CS25} remain valid in the present setting, allowing us to conclude the proof.
A similar approach was employed in \cite{C25-2} to derive an analogous result in the case $m=1$ and $p=1$ in \eqref{examplesf}. We note that the method developed in \cite[Section~6]{CS25} (the case of $m=0$ and $p>1$ in \eqref{examplesf}) follows the simplified version of the classical center manifold analysis for radially decreasing solutions developed in \cite{S19}, which was originally introduced in \cite{filippas1992refined}; see also \cite{HV93, Vel93b, bebernes1992final,WXL}.

With Theorem~\ref{theo} at hand, we derive the lower bound in \eqref{infer1} following an argument similar to that in \cite[Section~5]{CS25}, suitably adapted to the present context. 
The main idea of proof is, 
for $t_0$ suitably close to $T$, to use the lower estimate on $u(t_0)$ provided by Theorem~\ref{theo}
in the region $|x|=O(\sqrt{T-t_0})$ and propagate it forward in time via a delicate comparison argument.

Regarding the upper bound in \eqref{infer1}, it follows directly from \cite[Theorem~3.1]{CS25}, combined with Proposition~\ref{qssblowup} and the properties of $f$. It therefore remains to verify that the assumptions of \cite[Theorem~3.1]{CS25} are satisfied by the nonlinearities considered in Theorem~\ref{Thm}, and that the resulting estimates can be recast in the form required by Theorem~\ref{Thm}.

\section{Preliminaries}
In this section, we establish several properties of $f$, $F$ and there inverses that will be used in the next section.
	We first observe that, under assumption \eqref{sem0}-\eqref{defF}  and \eqref{qlim}, there exists $M>0$ such that 
		 \be{fposM}
		 f\in C^2([M,\infty))\quad\hbox{and}\quad f,f',f''>0\ \hbox{ on $[M,\infty)$}
		\ee 
		and $G:[M,\infty)\to(0,G(M)]$ is decreasing and $C^2$.

\begin{lemma}\label{lemmaFH}
Let $f$ and $G$ be as in \eqref{sem0}-\eqref{defF} with
$$
 f'(X)G(X)\to q\ge 1\quad\hbox{as }\ X\to \infty.
$$
We also define 
\be{Hqlim}
H(X):=\int_X^\infty \frac{A+\log f(s)}{f(s)}ds\quad\hbox{for large $X>0$ and some $A>0$}.
\ee
$(i)$ As $X\to \infty$, we have 
\be{propFf}
|\log G(X)|\sim \frac1q\log f(X)
\ee
and 
\be{propFH2}
H(X)\sim qG(X)|\log G(X)|.
\ee

\noindent$(ii)$ Moreover, as $Y\to 0^+$, we have 
\be{propinversFH2}
H^{-1}(Y)\sim G^{-1}\Big(\frac{Y}{q|\log Y|}\Big).
\ee

\end{lemma}

\begin{proof}[Proof of Lemma~\ref{lemmaFH}]
$(i)$ The proof of \eqref{propFf} and \eqref{propFH2} is based on the L'H\^{o}pital's rule. Indeed, using $f(s)\to \infty$ and $G(X)\to 0$ as $X\to \infty$, we have
$$
\begin{aligned}
\lim_{X\to \infty}\frac{\log G(X)}{\log f(X)}&=\lim_{X\to \infty}\frac{\frac{-1 }{f(X)G(X)}}{\frac{f'(X)}{f(X)}}=\lim_{X\to \infty}\frac{-1}{ f'(X)G(X)}=\frac{-1}{q}.
\end{aligned}
$$
Which yields the proof of \eqref{propFf}. Building on the latter, we derive \eqref{propFH2} by the same way. Namely, we compute:
$$
\begin{aligned}
\lim_{X\to \infty}\frac{G(X)|\log G(X)|}{H(X)}&=\lim_{X\to \infty}\frac{\frac{\log G(X)}{f(X)}+\frac{1}{f(X)}}{-\frac{\log f(X)}{f(X)}}=\lim_{X\to \infty}\Big(\frac{-\log G(X)}{\log f(X)}-\frac{1}{\log f(X)}\Big)=\frac{1}{q}.
\end{aligned}
$$

\noindent$(ii)$ We shall derive \eqref{propinversFH2} from \eqref{propFf}-\eqref{propFH2}. 
By \eqref{propFH2}, we first note that $\log H(X)\sim\log G(X)$ at $\infty$, and then 
$$
H(X)\sim qG(X)|\log H(X)|.
$$
Thus, writing $Y:=H(X)$, we get 
$$
G(H^{-1}(Y))\sim\frac{Y}{q|\log Y|}.
$$
Which implies that 
$$
H^{-1}(Y)=G^{-1}(G(H^{-1}(Y)))=G^{-1}\Big((1+o(1))\frac{Y}{q|\log Y|}\Big).
$$
This along with \eqref{equivGG} below (which is independant of \eqref{propinversFH2}) concludes the proof.
\end{proof}

\begin{lemma}\label{ell12a}
	Let $f$ satisfy the assumptions \eqref{sem0}-\eqref{qlim} and $F$ be defined by \eqref{defF}.
	\smallskip
	 There exist $C_0,C_1,X_1>0$ such that, if $0\le \eps<1/2$ then for all $X\ge X_1$, we have 
	
		\be{ell12a1}
f((1-\eps))X)\ge e^{-C_0\eps}f(X)	
		\ee
		and
		\be{ell12a2}
	G((1-\eps )X)\le (1+C_1\eps)G(X).
	\ee
\end{lemma}
\begin{proof}[Proof of Lemma~\ref{ell12a}] By remark~\ref{CSLem}, we refer to \cite[Lemma~7.1]{CS25} for the proof of this lemma in the case of $q>1$ and also \cite[Lemma~7.2]{C25-2} in the particular case of $m=1$ and $p=1$ in \eqref{examplesf}. We give here a proof using the convexity of $f/\log f$ and $f\to \infty$ at $\infty$.
To this end, let $X_1>0$ be such that 
$$
(f/\log f)''\ge 0,\quad\log f(X)\ge 2,\quad \hbox{and }\ f'\ge 0\quad\forall \ X\ge X_1.
$$ 

 Set $h(X)=f(X)/\log f(X)$. By the first part of \eqref{qlim}, we have $h,h''\ge0$ for $X\ge X_1$, and  hence 
$$
h(X/2)\le h(X)-\frac{X}{2} h'(X).
$$ 
This in particular implies that 
$$
h'(X)\le \frac2X (h(X)-h(X/2))\le \frac{2 h(X)}{X}.
$$
Thus by a simple computation for large $X$, we get 
$$
0\le \frac{f'(X)}{f(X)}=\frac{h'(X)}{h(X)(1-1/\log f(X))}\le \frac{2h'(X)}{h(X)}\le \frac{4}{X}.
$$
We are ready to complete the proof of \eqref{ell12a1} as follow:
$$
\begin{aligned}
\log \frac{f(X)}{f((1-\eps)X)}&=\int_{(1-\eps)X}^X\frac{f'(s)}{f(s)}ds\le 4 \int_{(1-\eps)X}^X\frac{1}{s}ds\\
&\le-4\log(1-\eps)\le 8\eps,
\end{aligned}
$$
where we obtain the last inequality by using $-\log(1-\eps)\le 2\eps$ for any $0<\eps<1/2$. Compose with $x\mapsto\exp(x)$, increasing function, we get
$$
 \frac{f(X)}{f((1-\eps)X)}\le e^{ 8\eps} \quad \hbox{for large }\ X.
$$
This conclude the proof of \eqref{ell12a1}.

By the definition of $F$ as in \eqref{defF}, \eqref{ell12a2} is a direct consequence of \eqref{ell12a1}.

\end{proof}

\begin{lemma}\label{lem1-1}
$(i)$ If a function $\eps$ satisfies $\lim_{Y\to 0^+}\eps(Y)=0$ then
 		\begin{equation}\label{equivGG}
 		\lim_{Y\to 0^+} \frac{G^{-1}\Big((1+\eps(Y))Y\Big)}{G^{-1}(Y)}=1
 			\end{equation}
 		and the similar properties hold for  $H^{-1}$.
 		
 		\smallskip
 		
$(ii)$ For any $\beta>0$, we have 
 		\be{eqen0}
 		\lim_{Y\to 0^+} Y^{\beta}(G^{-1}(Y))^\alpha=0,\quad\hbox{for any } \alpha\in (0,\beta/q).
 		\ee
\end{lemma}

\begin{proof}[Proof of Lemma~\ref{lem1-1}]
\noindent$(i)$  We observe that \eqref{equivGG} for $G$ is a direct consequence of \eqref{ell12a2} and decreasing property of $G$. Therefore, by using \eqref{propFH2} we also deduce the conclusion for $H$.

\noindent$(ii)$ Let $X=G^{-1}(Y)$ and $g(Y)=f(G^{-1}(Y))$. By a simple computation, we have 
$$
g'(Y)=f'(X)\frac{d X}{dy}=-f'(X)f(X),
$$
which implies that 
$$
-\frac{Yg'(Y)}{g(Y)}\to q\quad\hbox{as }\ Y\to 0^+.
$$
Hence, there exists $Y_0>0$ such that 
$$
\frac{g'(Y)}{g(Y)}\ge -\frac{q+1}{Y}\quad\hbox{for all }\  0<Y\le Y_0.
$$
Therefore, we have 
$$
\log g(Y)-\log g(Y_0)\le (q+1)\log \frac{Y_0}{Y} \quad\hbox{for all }\  0<Y\le Y_0.
$$
Thus 
$$
g(Y)\le  Y^{-(q+1)} \Big(Y_0^{(q+1)}g(Y_0)\Big) \quad\hbox{for all }\  0<Y\le Y_0.
$$
Now we use 
$$
\begin{aligned}
G^{-1}(Y)-G^{-1}(Y_0)&=-\int_Y^{Y_0}(G^{-1}(z))'dz=\int_Y^{Y_0}f(G^{-1}(z))dz=\int_Y^{Y_0}g(z)dz\\
&\le  Y_0^{(q+1)}g(Y_0)\int_Y^{Y_0}z^{-(q+1)}dz\le C(q,Y_0,g(Y_0))(Y^{-q}-Y_0^{-q})\le CY^{-q}
\end{aligned}
$$
Hence for any $\alpha,\beta,>0$ we have 
$$
Y^\beta(G^{-1}(Y))^\alpha\le CY^{\beta-\alpha q}+Y^\beta G^{-1}(Y_0) \quad \hbox{for all }\ 0<Y\le Y_0
$$
This conclude the proof of \eqref{eqen0}.
\end{proof}

\section{Proof of upper part of Theorem \ref{Thm} assuming Proposition~\ref{qssblowup}} \label{upperrrrrr}

The upper estimate of $u$ in Theorem \ref{Thm} i.e., 
\begin{equation}\label{infer1-upper}
	u(x,t)\le (1+o(1))G^{-1}\bigg(T-t+\frac{|x|^2}{4q|\log|x|^2|}\bigg),
	\end{equation}
 will be a consequence of the following result from \cite{CS25} (see \cite[Theorem $3.1$ and formula (4.22)]{CS25}). 

\begin{thm}\label{rappelH}
Let $\Omega=\mathbb{R}^n \text{ or }\Omega=B_R$, and $u_0\ge 0$ be radially symmetric, nonincreasing in $|x|$ and  $T<\infty$.
Assume that  
\begin{equation*}
 f\in C^1([0,\infty)),\quad f(0)\ge 0, \quad \lim_{s\to \infty}f(s)=\infty,
 \end{equation*}
 \begin{equation*}
	\hbox{$f$ is $C^2$ and } \frac{f}{\log f} \hbox{ is convex at $\infty$,} \quad \int^\infty \frac{\log f(s)}{f(s)}ds <\infty.
	\end{equation*}
	If $\Omega=\mathbb{R}^n$ assume also that $u_0$ is nonconstant. 
	Then there exist $A_0,\eta>0$ such that
	\begin{equation}\label{upperTheoL}
	u(x,t)\le H^{-1}\bigg(H(u(0,t))+\frac{|x|^2}{4}\bigg)\quad \text{in }\ B_\eta\times(T-\eta,T),
	\end{equation}
	where $H$ is define as in \eqref{Hqlim}.
Moreover, there exist constants $A>0$ such that for $M>0$ as in \eqref{fposM},
\be{Rappelupper}
-u_r\ge \frac{rf(u)}{2(A+\log f(u))},\quad \hbox{in }\ Q:=\big\{(x,t)\in B_{R/2}\times(T/2,T); \ u(x,t)\ge M\big\}.
\ee
\end{thm}

\smallskip

\smallskip

 	\begin{proof}[Proof of upper part of Theorem~\ref{Thm}] 
 Since $m(t):=u(0,t)=\|u(t)\|_\infty\to \infty$, we have  $G(m(t))\to0^+$ 
 as $t\to T$. 
 
 Applying property \eqref{propFH2} to  inequality \eqref{upperTheoL}, 
 we get
$$u(x,t)\le H^{-1}\left(H(m)+\frac{|x|^2}{4}\right)
 =H^{-1}\left((1+o(1))\left(qG(m)|\log(G(m))|+\frac{|x|^2}{4}\right)\right)$$
as $(x,t)\to (0,T)$.  Next applying \eqref{propinversFH2} and \eqref{equivGG},
we obtain
 \begin{equation}\label{nuitt}
 u(x,t)\le (1+o(1)) G^{-1}\left(\frac{G(m)|\log(G(m))|+\frac{|x|^2}{4q}}
 {\Bigl|\log\Bigl(G(m)|\log(G(m))|+\frac{|x|^2}{4}\Bigr)\Bigr|}\right).
\end{equation}
 Also, for $(x,t)$ close enough to $(0,T)$,
 we have $G(m)\le G(m)|\log(G(m))|+\frac{|x|^2}{4}<1$,
 so that
 $$\begin{aligned}
 \bigg|\log\bigg(G(m)|\log G(m)|+\frac{|x|^2}{4}\bigg)\bigg|
 &\le\min\left\{|\log(G(m))|,\bigl|\log\bigl(\textstyle\frac{|x|^2}{4}\bigr)\bigr|\right\}\\
 &=(1+o(1))\min\left\{|\log(G(m))|,\bigl|\log|x|^2\bigr|\right\}
 \end{aligned}$$
 as $(x,t)\to (0,T)$, hence
 $$\frac{G(m)|\log(G(m))|+\frac{|x|^2}{4q}}
 	{\Bigl|\log\Bigl(G(m)|\log(G(m))|+\frac{|x|^2}{4}\Bigr)\Bigr|}
 	\ge (1+o(1))\left(G(m)+\frac{|x|^2}{4q\bigl|\log|x|^2\bigr|}\right).$$
  Since $G^{-1}$ is decreasing near $0^+$, 
  we thus deduce from \eqref{nuitt} 
    that
 \begin{equation}\label{conclprop31}
  u(x,t)\le (1+o(1))G^{-1}\left((1+o(1))\left(G(m)+\frac{|x|^2}{4\bigl|\log|x|^2\bigr|}\right)\right),
 \quad (x,t)\to (0,T).
 \end{equation}
 On the other hand, by Proposition~\ref{qssblowup}, we have 
 \begin{equation}\label{0J1}
  G(u(0,t))=(1+o(1))(T-t) \quad\hbox{as }\quad t\to T.
  \end{equation}
  Combining this with \eqref{conclprop31} and applying~\eqref{equivGG} once again, we obtain  \eqref{infer1-upper} and so the upper part of~\eqref{infer1}.
 \end{proof}

\section{Proof of Proposition~\ref{qssblowup}}\label{GK}

{\bf Step 1.} {\it Gradient upper estimate and no-needle property.} We shall prove the following lemma.
\begin{lemma}\label{noneedle}
	Under assumptions of Proposition~\ref{qssblowup}, there exists $R_0>0$ and  $t_1 \in (\max\{T-1,0\},T)$ such that 
	\begin{equation}\label{eq:esth}
		u(R_0\sqrt{T-t}, t) \geq \frac12G^{-1}(T-t)\quad\ \text{for} \ t \in (t_1,T).
	\end{equation}
\end{lemma}
We will use the following gradient estimate in the proof of this lemma.
\begin{lemma}\label{grad-estim-1}
Under assumptions of Proposition~\ref{qssblowup}, there exists $\delta>0$ such that 
\be{grad-estim1}
\forall\ t\in(T-\delta^2,T),\quad\|\nabla u(\cdot,t)\|_\infty\le C_0 \begin{cases} (T-t)^{-1/2}\quad\hbox{if }\ q=1,\\ (T-t)^{-1/2}G^{-1}(T-t)\quad\hbox{if }\ q>1,\end{cases}
\ee
with $C_0:=C_0(n,f,T,\delta)>0$.
\end{lemma}
\begin{proof}[Proof of Lemma~\ref{grad-estim-1}]
 We refer to \cite[Proposition~3.2]{C25-1} for the case of $q>1$ due to Remark~\ref{CSLem}. Using \eqref{propFf}, the case of $q=1$ follows by the same approach as in \cite[Lemma~4.3]{C25-2}.
\end{proof}
 \begin{proof}[Proof of Lemma~\ref{noneedle}]
Fix $t_1\ge T-\delta^2$ and $R_0>0$ (to be choosing later). Using \eqref{grad-estim1} for $q>1$, we have

$$
u(0,t)-u(R_0\sqrt{T-t},t)=\int_0^{R_0\sqrt{T-t}}-u_r(r,t)dr\le C_0R_0G^{-1}(T-t).
$$
Therefore, by \eqref{belowrate} and choosing sufficiently small $R_0\le \frac{1}{2C_0}$, we get 
$$
u(R_0\sqrt{T-t},t)\ge (1-C_0R_0) G^{-1}(T-t)\ge \frac{1}{2}G^{-1}(T-t).
$$

The case of $q=1$ follows by the same way.
 \end{proof}

\vspace{.2cm}

\noindent{\bf Step 2.} {\it Rescaled equation.} We defined the similarity variables as
\be{dam111}
y=\frac{x}{\sqrt{T-t}}\quad\hbox{and }\ s=-\log(T-t).
\ee
Here we consider the PDE in \eqref{SHE} in $B_{R_0}\times (t_1,T)$, where $R_0>0$ and $t_1$ are defined in Lemma~\ref{noneedle}. 
Let $s_0>-\log(T-t_1)$ such that $G^{-1}(e^{-s})\ge M$ in $[s_0,\infty)$ where $M$ is defined as in \eqref{fposM}. We then introduce the following rescaled function:
\be{dam11}
w(y,s)=\begin{cases}
\log \Big(\frac{T-t}{G(u(x,t))}\Big)\quad \hbox{if }\ q=1\\
\Big(\frac{(q-1)(T-t)}{G(u(x,t))}\Big)^{q-1} \quad \hbox{if }\ q>1.
\end{cases}
\ee
Since $u$ is a type I solution of \eqref{SHE}, it the follows from \eqref{grad-estim1} that 
\be{DwysM1}
w\le M,\quad|\nabla w|\le M, 
\ee
for some $M>0$. Which in particular implies in the case of $q=1$ that 
\be{boundedgradu}
-M(1+|y|)\le w(y,s)\le M,\quad |\nabla w|\le M.
\ee

\begin{itemize}

\item {\bf Equation for $q=1$.} Setting
$ u(x,t) = G^{-1}\left(\frac{T-t}{v(y,s)}\right)$, i.e., $ v(y,s)=\frac{T-t}{G(u(x,t))}$, we have 
\begin{equation}\label{veq}
		v_s = \Delta v -\frac{y}{2} \cdot \nabla v - \frac{|\nabla v|^2}{v} + v^2 - v + \frac{|\nabla v|^2}{v}\Bigl(f'(u)G(u) - 1\Bigr). 
	\end{equation}
	This was already done in \cite{Ryoto}, we give here a proof for convenience of readers.
It follows by a simple computation that
 \begin{align*}
	u_t &= f(u)\left[\frac{1}{v} + \frac{T-t}{v^2}\left(\frac{v_s}{T-t} + \frac{y}{2(T-t)}\cdot \nabla_y v\right) \right],\\
	\nabla_x u &= f(u)\frac{\nabla_y v}{v^2}\sqrt{T-t},\\
	\lap_x u &= f'(u)\frac{\nabla_x u \cdot \nabla_y v}{v^2}\sqrt{T-t} + f(u)\left(\frac{\lap_y v}{v^2} - 2\frac{|\nabla_y v|^2}{v^3}\right)\\
				 &= f'(u)f(u)G(u) \frac{|\nabla_y v|^2}{v^3} + f(u)\left(\frac{\lap_y v}{v^2} - 2\frac{|\nabla_y v|^2}{v^3}\right).
 \end{align*}
 We then obtain \eqref{veq} by using $u_t-\Delta u=f(u)$.
Now using $v=e^w$, we obtain
\be{eqw}
w_s -\Delta w +\frac{y}{2} \cdot \nabla w = e^w - 1 + |\nabla w|^2\Bigl(f'(u)G(u) - 1\Bigr).
\ee 

\item {\bf Equation for $q>1$.} By the same way, a simple calculation yields
\be{eqwgrand}
w_s -\Delta w +\frac{y}{2} \cdot \nabla w =w^\frac{q}{q-1} -(q-1) w + \frac{|\nabla w|^2\Bigl(f'(u)G(u) - q\Bigr)}{(q-1)w}.
\ee
\end{itemize}
Considering the PDE in \eqref{SHE} on $B_{R_0}\subset\mathbb{R}^n$, we see that $w$ is a global-in-time solution of 
\begin{equation}\label{VraieE0}
 w_s-\Delta w +\frac{1}{2}y\cdot\nabla w=f_1(w)+ H_0(s,y)
\end{equation}
which is equivalent to
\begin{equation}\label{VraieE}
\rho w_s-\nabla\cdot(\rho\nabla w)=f_1(w)\rho +\rho H_0(s,y)\quad \hbox{in }\  \mathcal{W}:=D(s)\times (s_0,\infty),
\end{equation}
where 
$$ D(s):=e^{s/2}B_{R_0},$$
 the main terms $f_1$ and remainder terms  $H_0$ are given by:

$\bullet$ if $q=1$:
\begin{equation}\label{VraieEE}
f_1(w)=e^w-1\quad \hbox{and }\quad H_0(s,y):=|\nabla w|^2\Bigl(f'(u)G(u) - 1\Bigr);
\end{equation}

$\bullet$ if $q>1$:
\begin{equation}\label{VraieEEmod}
f_1(w)=w^\frac{q}{q-1} -(q-1) w\quad \hbox{and }\quad H_0(s,y):=\frac{|\nabla w|^2\Bigl(f'(u)G(u) - q\Bigr)}{(q-1)w}.
\end{equation}

\vspace{.2cm}

\noindent{\bf Step 3.} {\it Control of nonautonomous terms $H_0$.}

By \eqref{eq:esth}, using the radial decreasing property of $u$, we have 
$$
u(y\sqrt{T-t},t)\ge  \frac12G^{-1}(T-t) \quad \hbox{for $y\in B_{R_0}$ and $t>t_1$}.
$$
It then follows, by assumption \eqref{qlim} and properties  \eqref{dam1a}, \eqref{grad-estim1} and \eqref{DwysM1}, that 
\be{decayH}
\|H(s,\cdot)\|_{L^\infty(D(s)}\le C s^{-\alpha_0}\quad\hbox{for all }\ s\ge s_0,
\ee
 where we use the fact that
$$
\nabla_y w=(q-1)w \frac{(T-t)^{1/2}\nabla_x u}{f(u)G(u)} \quad \hbox{for }\ q>1.
$$

\vspace{.2cm}

{\bf Step 4.} {\it Required regularity of $w$.}
 For all $0<\eps<\tau< T$, denoting $Q_{\eps,\tau}:=\R^n\times(\eps,\tau)$, it follows from standard parabolic $L^q$ and Schauder regularity  
 (see, e.g., \cite[Section~48.1]{quittner2019superlinear}) that 
 \be{utbound}
 u_t,D^2u\in BC(\bar Q_{\eps,\tau})
 \ee
 (where $BC$ denotes the set of bounded continuous functions).
 Then, setting $g:=f'(u)u_t\in L^\infty(Q_{\eps,\tau})$,
the function  $Z:=u_t$ is the solution of the problem $Z_t-\Delta Z=g$ in $Q_{\eps,\tau}$ with $Z=0$ on $\partial\Omega\times(\eps,\tau)$
  and $Z(\eps)=u_t(\eps)$. By parabolic $L^q$ regularity, we deduce that, for all $1<q<\infty$ and $0<\eps<\tau< T$,
  $$\sup_{x_0\in\Omega} \|Z_t\|_{L^q(Q_{\eps,\tau,x_0})}+\|D^2Z\|_{L^q(Q_{\eps,\tau,x_0})}<\infty,$$
  where $Q_{\eps,\tau,x_0}=(\Omega\cap B_1(x_0))\times(\eps,\tau)$ (the set $Q_{\eps,\tau,x_0}$ can be replaced by $Q_{\eps,\tau}$
  and the supremum over $x_0$ omitted
   in case $\Omega$ is bounded). By interpolation inequalities it follows that
 \be{nablautbound}
 \nabla u_t\in BC(\bar Q_{\eps,\tau}).
 \ee
 We deduce from \eqref{dam11}, \eqref{DwysM1}, \eqref{utbound} and \eqref{nablautbound} that, for all $s_1\in(s_0,\infty)$, 
 \be{eq:estderivative}
D^2w,\ (1+|y|)^{-1}w_s,\ (1+|y|)^{-1}\nabla w_s\in BC(\overline{\mathcal{W}}_{s_1}),
\ee
where  $\mathcal{W}_{s_1}=\mathcal{W}\cap(\R^n\times(s_0,s_1))$.
In view of the exponential decay of $\rho$, this guarantees the convergence and the differentiability of the
various integrals and justifies the integrations by parts in the rest of the proof as well as in Section~\ref{LOT5} below.

\vspace{.2cm}

\noindent{\bf Step 5.} {\it Weighted energy functional and its properties.}
A main advantage of our transformation $w=\log v$, is that the Liapunov functional for equation  \eqref{VraieE} will be given by the    weighted energy  defined as follows: 
\begin{equation}\label{LO2}
\E[w](s):= E[w](s)+C_1 s^{-\gamma}
\end{equation}
where  $\gamma=2\alpha_0-1$, $C_1>0$ is a  constant, and 
$$
E[w]:=\int_{D(s)}\bigg(\frac{1}{2}|\nabla w|^2+F_1(w)\bigg)\rho dy,
$$
with
 $$
 F_1(w)=\begin{cases}
 w-e^w \quad\hspace{3cm}\hbox{if }\ q=1, \\
 \frac{q-1}{2}w^2-\frac{q-1}{2q-1} w^\frac{2q-1}{q-1}\quad\hspace{1cm}\hbox{if }\ q>1. 
 \end{cases}
 $$
  It appear as a perturbation of the weighted energy in the case of a pure exponential/power nonlinearity. Another important point is that \eqref{VraieE} is similar to the corresponding expression in \cite[Formula 4.22]{C25-2} for $q=1$ and \cite[Formula~3.28]{C25-1} in the case of $q>1$, except for the definition of $H$. 
Following the approach in \cite[Lemma~4.5]{C25-2} and adapting it to our context, we obtain the following properties of $\E$. Indeed, since $|y|=R_0e^{s/2}$ for all $y\in \partial D(s)$, the exponential decay of $\rho$ guarantees the required control of all boundary integrals. Moreover, using property \eqref{decayH} together with the exponential decay of $\rho$, we obtain the desired control of the terms involving $H$.

\begin{lem}[Properties of $\E$ for $q\ge1$]\label{ell7}
	For all $s>s_0$, there exists $M_0\ge0$ (independent of $s,y$) such that 
	\begin{equation}\label{jojo1}
	\frac{d}{ds}\E[w](s)\le-\frac{1}{2}\int_{D(s)}w_s^2\rho\le 0
	\end{equation}
	
	and 
	
	\begin{equation}\label{JO1}
	\E[w](s)\ge -M_0.
	\end{equation}
	
\end{lem}

\vspace{.2cm}

\noindent{\bf Step 6.} {\it Proof of \eqref{eq:qssblowup}.}
By \eqref{jojo1}-\eqref{JO1}, we have
	\begin{equation}\label{ellGw}
	 \ell:=\lim_{s\to\infty} \E[w](s)\in [-M_0,\infty).
	\end{equation}
	  Pick any sequence $s_j\to \infty$ and set $z_j:=w(y,s+s_j)$.
By \eqref{boundedgradu}, we have  $-c(1+|y|)\le z_j(y,s)\le M$ and $|\nabla z_j(y,s)|\le M$, and
	it follows from 
\eqref{VraieE0}, \eqref{decayH} and parabolic estimates that
	the sequence $(z_j)_j$ is precompact in $C^{2,1}_{loc}\big(\mathbb{R}^n\times[0,1]\big)$.
	Consequently, there exists a subsequence (still denoted $s_j$) and a solution $\phi$ of
\begin{equation}\label{Vraie00}
 \phi_s-\Delta \phi +\frac{1}{2}y\cdot\nabla \phi=f_1(\phi) \quad \hbox{in } \mathbb{R}^n\times[0,1]
\end{equation}
such that $w(\cdot,\cdot+s_j)\to \phi$ in $C^{2,1}_{loc}\big(\mathbb{R}^n\times[0,1]\big)$. 
	Moreover 
$|\nabla \phi|$ is bounded in $\mathbb{R}^n\times[0,1]$ and  $-c|y|\le \phi(y,s)\le M$. Let $ K>0$. 
	 Using \eqref{jojo1}, \eqref{ellGw} and the fact that $B_{ K}\subset D(s)$ for all sufficiently large $s$, we deduce that
	\begin{equation*}
	\int_{0}^{1}\int_{B_{ K}}\big(\partial_s z_j\big)^2\rho dy ds\le 
	 \int_{s_j}^{s_j+1}\int_{D(s)}\big(\partial_s w\big)^2\rho dyds
		\le 2\E[w](s_j)-2\E[w](s_j+1)\to 0,\quad j\to \infty.
	\end{equation*}
	By Fatou's lemma, and since $ K>0$ is arbitrary,
	we deduce that $\partial_s \phi\equiv 0$ in $\R^n$. Therefore, by \eqref{Vraie00}, 
	we have $\Delta \phi-\frac12 y\cdot\nabla \phi+f_1(\phi)=0$ in $\mathbb{R}^n$.  
	Moreover, since  $u$ is radially symmetric, so is $\phi$. 
	
	\smallskip
	
\noindent{\bf Step 6.} {\it Nondegeneracy of blow-up for $q=1$.} In dimension $n\le 2$, using the type~I property of $u$ together with \cite[Lemma 3.1, Theorem 3.2]{L89} (see also \cite[Lemma~2.8]{Ryoto}), we obtain $\phi\equiv0$.
In the case where the time monotonicity assumption holds, \cite[Lemma~4.7]{C25-2} implies the type~I blow-up rate for the solution. Passing to the limit, we obtain
$$
\phi(0)\ge 0,\qquad \phi'(|y|)\le 0,\qquad \text{and}\qquad 1+\frac12 y\cdot\nabla\phi\ge0.
$$
Hence, $\phi\equiv0$ by \cite[Lemma~4.6]{C25-2}.	
This concludes the proof of Proposition~\ref{qssblowup} for $q=1$.

\vspace{.2cm}

\noindent{\bf Step 7.} {\it Nondegeneracy of blow-up for $q>1$.} Setting $p=q/(q-1)$ we then have 
 
\be{eqlimw}
\Delta \phi-\frac12y\cdot\nabla \phi+\phi^p-\frac{1}{p-1}\phi=0\quad y\in \R^n.
\ee
On the other hand, using $q>q_S$ or $u_t\ge 0$, it follows by \eqref{belowrate} and \eqref{DwysM1} that
\be{hypothutp}
\begin{aligned}
&\phi(0)\ge (1/(p-1))^{1/(p-1)},\  \phi'(|y|)\le 0\ \hbox{and},\ 1<p<p_S\ \hbox{ or }\ \phi+\frac12 y\cdot\nabla\phi\ge0.
\end{aligned}
\ee
\begin{lemma}\label{lim}
Any classical solution of \eqref{eqlimw}-\eqref{hypothutp} is constant and 
\be{sollim}
\phi\equiv \kappa=(1/(p-1))^{1/(p-1)}=(q-1)^{q-1}.
\ee
\end{lemma}

\begin{proof}[Proof of Lemma~\ref{lim}]
If $1<p<p_S$, the Pohozaev identity \cite[Proposition~2]{GK85} implies that the solution is constant. Therefore, $\phi$ satisfies \eqref{sollim}, due to $\phi(0)\ge\kappa$.
It remains to prove the result in the case $u_t\ge0$ for $q>1$, i.e.,
$$
g(r):=\phi(r)+\frac12 r\phi'(r)\ge0,\quad\hbox{ for any }\ r=|y|\ge0.
$$
Suppose, by contradiction, that $\phi(0)>\kappa$. By uniqueness, we then have $\phi>\kappa$ on $(0,\infty)$, and hence
$$
0<f_1(\phi)\le C(p).
$$
Set 
$$
\mu(r)=r^{n-1}e^{-r^2/4}.
$$
 Then the equation  in \eqref{eqlimw} can be written as
$$
(\mu \phi')'(r)=-\mu (r)f_1(\phi(r)).
$$
By integrating, we get 
$$
\mu(r) \phi'(r)=-\int_0^r\mu (s)f_1(\phi(s))ds.
$$
Now using the fact that 
$$
\lim_{r\to \infty}\int_0^r\mu (s)f_1(\phi(s))ds=\int_0^\infty \mu(s)f_1(\phi(s))ds=:A\in (0,\infty);
$$
there exists $R_0>0$ such that 
$$
\phi'(r)\le -\frac{A}{2\mu(r)}=-\frac{Ae^{r^2/4}}{2r^{n-1}}\quad\hbox{for all } r\ge R_0.
$$
This along with the definition of g yields
$$
0\le g(r)=\phi(r)+\frac{r}{2}\phi'(r)\le \phi(0)-\frac{Ae^{r^2/4}}{4r^{n-2}}\underset{r\to \infty}{\to} -\infty.
$$
This yields the contradiction, hence $\phi\equiv \kappa$.
\end{proof}

\section{Proof of Theorem \ref{theo}} \label{LOT5}

Recalling definition \eqref{defL2H1}, we respectively denote  by $(v,w):=\int_{\mathbb{R}^n}vw\rho dy$ and $\| v\|=(v,v)^{\frac{1}{2}}$ the inner product and the norm of the Hilbert space $L^2_\rho$. 
The following proposition follows from the argument in \cite[Section~6]{CS25}.

\begin{prop}\label{TheoremBreduit}
Let $z$ be a radially symmetric  nonincreasing solution of  
$$
\partial_s z+\mathcal{L}z-z = \mathcal{N}(z,s)+\tau(s)z+B(y,s)\quad \hbox{in $\ \R^n\times(s_1,\infty)$,}$$ 
 where 		
$$\tau(s)=o\Bigl(\frac{1}{s}\Bigr),\quad s\to\infty,$$

$$		|B|+|\nabla B|\le Ce^{-s},\quad |\mathcal{N}(\xi,s)| \le C\xi^2,  \quad
		|\partial_\xi \mathcal{N}(\xi,s)|\le C|\xi|,\quad s\ge s_1,$$
		and 
		$$V(s):=\lim_{\xi\to 0}\frac{\mathcal{N}(\xi,s)}{\xi^2}=k+K(s),\quad\hbox{with $k>0$ and } \lim_{s\to \infty}K(s)=0.$$
		We also assume that there exists $M_0> 0$ such that 
		\begin{equation*}
	 |z(y,s)| \le M_0(1+|y|),\quad  |\nabla z|\le M_0 \quad\hbox{ and }\quad 
	\lim_{s\to \infty} z(y,s)=0\ \text{uniformly for } y \text{ bounded,} 
	\end{equation*}
	and that
	$$
	\|z(.,s)\| \geq cs^{-1}, \quad  s\ge  s_1.  
	$$
	Then 
	\begin{equation}\label{control1-2-1}
	z(y,s)=\frac{ 2n-|y|^2}{8ks}+o\Bigl(\frac{1}{s}\Bigl)\quad \text{as } s\to \infty,
	\end{equation}
	with convergence in $H^1_\rho(\mathbb{R}^n)$ and uniformly for $y$ bounded.
\end{prop}
We shall apply Proposition~\ref{TheoremBreduit} for a suitable new function $B$ arising from our transformation. To this end, we use the same notation as in \eqref{dam111}-\eqref{dam11}.

\subsection{Nonexponential decay property}

 As in the previous works \cite{C25-2,CS25,S19}, we first provide a 
 lower bound on the decay of $w$ in $L^{2}_{\rho}$, 
 which is one of the key assumptions in Proposition \ref{TheoremBreduit}, as consequence of a lower bound on $|u_r|$ (cf.~\eqref{Rappelupper}).

\begin{lemma}\label{lem}
	Let $\kappa=(q-1)^{q-1}$ with $\kappa=0$ if $q=1$. Under the assumptions of Theorem~\ref{theo}, there  exists $c > 0$ 
	such that
	\be{lowerdecayw}
	\|\kappa-w(.,s)\| \geq cs^{-1}, \quad  s\ge  s_1.  
	\ee
	
\end{lemma}

\begin{proof} Here we deal with the case $q=1$. The case of $q>1$ follows by the same approach, see \cite[Lemma~6.2]{CS25}.
	
	Fix $ K > 0$ and $\eps>0$. By \eqref{eq:qssblowup} there exists $t_0(K,\eps) \in (T/2, T)$ such that  
	$$
M\le (1-\eps)\psi(t)\le	u(r,t) \le (1+\eps) \psi(t), \quad 0\leq r \leq K\sqrt{T-t}, \ \ t_0 < t < T,
	$$
	where $M$ is as in   \eqref{Rappelupper}.  
	Then, from \eqref{Rappelupper},  $u= u(r,t)$ with  $r=| x |,$ satisfies (see also~\cite[formula $(4.22)$]{CS25})
	\begin{equation}\label{(5.10)}
	-u_r(r,t) \geq\frac{rf(u)}{2(A +\log f(u)) },\quad  0 \leq r \leq K\sqrt{T-t}, \ \ t_0 < t < T.
	\end{equation}
	 We write $w =w (\eta, s)=\log v$ with $\eta =  |y|$ and $\varphi=-w$.
	 Using the properties of $L$, we deduce that, for all 
	$s_2 = s_2(K)\ge  \max(s_0,1) $ sufficiently large, there exists  $c_1=c_1(\eps)>0$,
	\be{roi}
	\begin{aligned}
		\varphi_\eta(\eta, s) &:= -w_\eta (\eta, s) =-\frac{v_\eta}{v} =-\sqrt{T-t} \,\frac{ u_r(\eta \sqrt{T-t},t)}{f(u)G(u)} \\
		&\ge \frac{ \eta(T-t)}{2\bigl(A + \log f(u)\bigr)G(u)}\ge C
		\frac{ \eta}{A + \log f(\psi(t))}\times \frac{ T-t}{G(u)}\ge C \frac{\eta}{\log f(\psi(t))}.
	\end{aligned}
	\ee
	 Using \eqref{propFf}, we have $\log(f(\psi(t)))\le \log f(G^{-1}(T-t))\le C|\log(T-t)|$ as $t\to T$.
		This combined with  \eqref{roi} implies
	\begin{equation*}
	\varphi_\eta(\eta, s)\ge \frac{c\eta}{s}\quad\hbox{for all $0 \leq \eta \leq K$ and $s > s_2$},
	\end{equation*}
	where the constant $c > 0$ is independent of $K$.
	
	Now, choose $K = 2(1 + c^{-1})$ and take any $s \ge s_2(K).$ If $ \varphi(1,s) \geq -1/s$ then it follows that 
	\begin{equation*} \varphi(\eta, s) = \varphi(1,s ) + \int_{1}^{\eta} \varphi_\eta(z,s)dz \geq \frac{-1 + c(\eta -1)}{s} \geq \frac{1}{s}, \ \ K-1 \leq \eta \leq K, \end{equation*}
	and hence $\|\varphi(\cdot,s)\| \geq \bigl(\int_{K-1\leq | y | \leq K} \rho\bigr)^{1/2}s^{-1}$. Otherwise, we have $\varphi(1,s) \leq -1/s$ and, since $\varphi$ is a nondecreasing function of $\eta$, we get $\varphi(\eta, s) \leq -1/s$ for $\eta \in [0,1];$ hence $\|\varphi (\cdot,s)\| \geq \bigl(\int_{| y | \leq 1} \rho\bigr)^{1/2}s^{-1}.$ We conclude that $\|\varphi (\cdot,s)\| \geq cs^{-1}$ for all $s \geq s_2.$
\end{proof}

\subsection{ Proof of Theorem~\ref{theo}}
In this case, we need another change of function as follows:
\be{defphi}
\varphi(y,s)=\kappa-w(y,s).
\ee
Hence the conclusion of Theorem~\ref{theo} becomes
$$
\varphi(y,s)=\beta(q)\frac{|y|^2-2n+o(1)}{4s} \quad\hbox{as }\ s\to \infty,\quad \beta(q)=\begin{cases}1\quad\hbox{if }\ q=1,\\\frac{(q-1)^q}{q}\quad\hbox{if }\ q>1;\end{cases}
$$
with convergence in $H_\rho^1(\R^n)$ and uniformly for bounded $y$.
To this end, we follow the same argument as in previous work \cite{CS25}.

\begin{proof}
Integrals over $\mathbb{R}^n$ will be simply denoted by $\int$, and the variables will be omitted when no confusion arises.

\noindent\textit{\textbf{Step 1.} Bounds and equation for $\varphi$.} 
The function
 $\varphi $ is radially symmetric nonincreasing, and
\be{boundvarphiq}
 |\nabla \varphi|\le M_0,\quad\hbox{and }\quad \begin{cases}	 - M_0\le w(y,s)\le  M_0(1+|y|)\quad\hbox{if }\ q=1,\\ -M_0\le \varphi(y,s)|\le  \kappa,\quad\hbox{if }\ q>1;
 \end{cases}
\ee
and  by Proposition~\ref{qssblowup},
\be{(1.9)q}
\lim_{s\to \infty} \varphi(y,s)=0\quad \text{uniformly for } y \text{ bounded.}
\ee
Therefore, for each $R > 0$, we have
\begin{equation}\label{(5.13)q} 
m_R(s):= \underset{| y | \leq R}\sup |\varphi(y,s)| \to 0\quad \text{as}\ \ s\to \infty,\end{equation}
and 
	\be{varphi0q}
\underset{s\to\infty}{\lim}\| \varphi(\cdot,s) \| = 0,
\ee
as a consequence of \eqref{(1.9)q} and dominated convergence. 
Moreover, by recalling \eqref{VraieE0}, we can rewrite the equation in the following form:
 \be{(5.11)q} 
 \varphi_s + \mathcal{L}\varphi-\varphi =\mathcal{N}(\varphi) - R(\varphi),
 \quad
y\in\R^n,\ s\ge s_1,
\ee
with
\be{defWTq}
\mathcal{N}(\xi):=-f_1(\kappa-\xi) -\xi, \quad R(\varphi)=H_0(\kappa-\varphi).
\ee
 We note that 
\be{(5.14)q} 
|\mathcal{N}(\varphi)| \le C\varphi^2,  \quad
|\partial_\xi \mathcal{N}(\varphi)|\le C|\varphi|,\quad y\in\R^n,\ s\ge s_1
\ee
and 
\be{cvFVq}
\lim_{\xi\to 0}\frac{\mathcal{N}(\xi)}{\xi^2}=-\frac12\beta(q).
\ee
This is a consequence of Taylor's formula with integral remainder:
\be{TaylorRIq} 
\mathcal{N}(\xi)=-
\xi^2\int^{1}_{0} \partial_{\xi}^2 f_1(\kappa-t\xi)(1-t)dt.
\ee

\noindent\textit{\textbf{Step 2.} Control of  $R(\varphi)$}.
A key new fact is the following property of $R(\varphi)$, which follows from \eqref{qlim}, \eqref{DwysM1}-\eqref{boundedgradu} and  \eqref{eq:estderivative}.

\begin{lemma}\label{controlRL}
 We have
\be{boundRL}
\int \bigl(R^2(\varphi)(1+|y|^4)+|R(\varphi)|+|\nabla R(\varphi)|^2\bigr)(y,s) \rho\le \eps(s)\int|\nabla \varphi|^2\rho\quad\hbox{as }s\to \infty,
\ee
where $\lim_{s\to \infty}\eps(s)=0$.
\end{lemma}

\begin{proof}[Proof of Lemma~\ref{controlRL}]
We will focus on the proof of \eqref{boundRL} for $q=1$, since the case of $q>1$ follows by the same way.
Recall that for any $W\in L^1(\R^n)$, we have 
$$
\lim_{R\to \infty}\int_{|y|\ge R}|W(y)|dy=0.
$$
Fix $\eps>0$ and recall \eqref{DwysM1}-\eqref{boundedgradu} and \eqref{eq:estderivative}. There exists $R_0>0$ such that 

$$
\begin{aligned}
\int_{|y|\ge R_0} \bigl(R^2(\varphi)+|R(\varphi)|\bigr)(y,s) \rho&\le \eps\int \bigl(R^2(\varphi)+|R(\varphi)|\bigr)(y,s) \rho\\
&\le C\eps\int |\nabla \varphi|^2 \rho
\end{aligned}
$$
 By Proposition~\ref{qssblowup} there exists $t_0(R_0,\eps) \in (T/2, T)$ such that  
	$$
\max\{M,X_0\}\le (1-\eps)\psi(t)\le	u(r,t) \le (1+\eps) \psi(t), \quad 0\leq r \leq R_0\sqrt{T-t}, \ \ t_0 < t < T,
	$$
	where $M$ is defined in \eqref{Rappelupper}. By \eqref{qlim}, we have 
$$
\begin{aligned}
\int_{|y|\le R_0} \bigl(R^2(\varphi)+|R(\varphi)|\bigr)(y,s) \rho&\le Cf_0(\psi(t))^2\int |\nabla \varphi|^2 \rho
\end{aligned}
$$
Hence
$$
\int \bigl(R^2(\varphi)+|R(\varphi)|\bigr)(y,s) \rho \le C\eps\int |\nabla \varphi|^2 \rho.
$$
By a simple computation, we have 
$$
\nabla R(\varphi)=2(f'(u)G(u)-1)D^2 \varphi \nabla \varphi+|\nabla \varphi|^2 \Big(f''(u)f(u)G^2(u)-f'(u)G(u)\Big)\nabla \varphi.
$$
It follows by the same argument as above, using \eqref{qlim}, that 
$$
\int|\nabla R(\varphi)|^2\rho\le C\eps\int |\nabla \varphi|^2\rho.
$$
It remains to prove that 
$$
\int R(\varphi)^2|y|^4\rho\le \eps \int |\nabla \varphi|^2\rho.
$$
We first note that there exists $R_\eps>0$ such that 
$$
\int_{|y|\ge R_\eps} R(\varphi)^2|y|^4\rho\le \eps \int R(\varphi)^2\rho\le \eps \int |\nabla \varphi|^2\rho.
$$
By \eqref{qlim}, choosing $s_1$ such that $R_\eps^4<s_1^{\alpha_0}$, we have 
$$
\begin{aligned}
\int_{|y|\le R_\eps} R(\varphi)^2|y|^4\rho&\le  R_\eps^4\int_{|y|\le R_\eps}  R(\varphi)^2\rho\\
&\le CR_\eps^4 s^{-2\alpha_0}\int |\nabla \varphi|^4 \rho\\
&\le C s^{-\alpha_0}\int |\nabla \varphi|^2\rho.
\end{aligned}
$$
This concludes the proof of Lemma~\ref{controlRL}.
\end{proof}

\noindent\textit{\textbf{Step 3.} Conclusion}.
Using \eqref{boundvarphiq}-\eqref{cvFVq} and Lemma~\ref{controlRL}, one easily checks that the arguments in \cite[Section~6]{CS25} remain valid after replacing $B$ by $R(\varphi)$. Hence, the conclusion follows from Proposition~\ref{TheoremBreduit} with $k=\frac12\beta(q)$, in view of \eqref{cvFVq}.

\end{proof}

\section{Proof of the lower part of Theorem~\ref{Thm}}

We shall derive 
	the lower part of Theorem \ref{Thm}:
	\begin{equation}\label{infer1-lower}
	u(x,t)\ge (1+o(1)) G^{-1}\bigg(T-t+\frac{|x|^2}{4q|\log|x|^2|}\bigg),
	\end{equation}
	as a consequence of
	  Theorem \ref{theo}.
	To this end we modify the argument in \cite[Section~5]{CS25} for the case of regular variation functions $f(u)=u^pL(u)$.
	 	\smallskip
		
 Let $ \delta>0$ (to be chosen below). 
Following \cite{C25-2,CS25}, we shall use the following properties of
$(S_\delta(t))_{t\ge0}$, the Dirichlet heat semigroup on
$B_\delta\subset\mathbb{R}^n$ (see \cite[Lemma~5.1]{CS25} for the proof).

 \begin{lemma}\label{lemSrho} 
 	(i) There exists a constant  $\lambda=\lambda(n,\delta)>0$ such that, for any $K,N>0$,
 	\begin{equation}
 	\Big[S_\delta(t) \bigl(N-K|x|^2\bigr)\Big]_+\ge e^{-\lambda t}\big[N-K(|x|^2 +2nt)\big]_+
 	\ \hbox{ in $Q:=\overline B_{\delta/2}\times[0,\infty)$.}
 	\end{equation}
 	
 	(ii) For all $\phi\in L^\infty(B_\delta)$, we have
 	\begin{equation}
 	 	\bigl|(S_\delta(t)\phi)(x)\bigr|\le C(n)t^{-n/4}\bigg(\int_{B_\delta}|\phi(z)|^2e^{-\frac{|z|^2}{4t}}dz\bigg)^{1/2}e^{\frac{|x|^2}{2t}},\quad x\in B_\delta, \ t>0.
 \end{equation}
\end{lemma}

\smallskip

 \begin{proof}[Proof of \eqref{infer1-lower}]
{\bf Step 1.} {\it Comparison argument.} 
 There exists $\delta\in(0,1)$ with $\delta<\min\{\sqrt{T},G(M)\}$,  such that
		\begin{equation}\label{uparabbdry}
		u(x,t)\ge  \tilde M:=G^{-1}\big(G( M)-\delta\big)> M \ \hbox{ in $Q:=\overline B_\delta\times[T-\delta^2,T)$,}
			\end{equation}
		 where  $M$ is a constant as in \eqref{fposM}.
	
	Fix any $\sigma \in(0,\delta^2)$ (to be chosen later) and set 
		\begin{equation}\label{deft0}
		t_0=T-\sigma.
			\end{equation}
 As in \cite{C25-2,CS25} (see also \cite{FM85}), we shall prove that 
\begin{equation}\label{a1}
	u\ge U\quad  \hbox{in} \quad B_\delta\times[t_0,T),
\end{equation} 
 where
	\begin{equation}\label{defUV}
	U=U_\sigma(x,t)= G^{-1}\Bigl\{G(V)+t_0-t\Bigr\},
	\quad\hbox{ with }V(x,t):= A+S(t-t_0)\bigl(u(t_0)- A\bigr).
	\end{equation}	
	Since $u(t_0)\ge  M$ in $\overline B_\delta$, we have $0\le S(t-t_0)(u(t_0)- M)\le \|u(t_0)\|_\infty- M$, so that $ M\le V\le \|u(t_0)\|_\infty$, hence $\tau_0:=G(\|u(t_0)\|_\infty)\le G(V)\le G( M)$. Letting 
	$$ E:=\{\tau>t_0; G(V)+t_0-t>0,\ \hbox{on } \bar{B}_{\delta}\times[t_0,\tau)\};\  \tau_1=\sup E,\ Q_1:=\bar{B}_\delta\times[t_0,\tau_1),$$
	we see that $E$ is nonempty, $\tau_1\ge t_0+\tau_0$, $U$ is well defined and smooth in $Q_1$.
By the same way as in \cite[section $5$ Step $2$]{CS25}, using \eqref{fposM}, we have
	$$\partial_tU-\Delta U\le f(U) \quad\hbox{in $\ Q_1$.}$$
Moreover, setting $T_1=\min\{T,\tau_1\}$, on $\partial B_\delta\times[t_0,T_1)$ we have $V= M$, hence $G(V)+t_0-t\ge G( M)-\delta$, so that $U \le G^{-1}\big(G( M)-\delta\big)= \tilde M$. Since also $U(t_0)=u(t_0)$, it follows from the comparison principle that  $u\ge U$ in $B_\delta\times [t_0,T_1)$ hence in particular $\tau_1\ge T$ (since otherwise $u$ would blow up before $t=T$).

\vspace{.2cm}

\noindent{\bf Step 2.} {\it Propagation of lower estimate on $u(t_0)$ in the region $|x|=O(\sqrt{T-t_0})$.}
		By Theorem~\ref{theo}, we have
		$$	 u(y\sqrt{\sigma},T-\sigma)=\begin{cases}G^{-1}\Big(\sigma\exp{\Big(\frac{|y|^2-2n}{4|\log\sigma|}
		+\frac{\mathcal{R}(y,|\log \sigma|)}{|\log\sigma|}\Big)}\Big)\quad\hbox{if }\ q=1\\
		G^{-1}\Big[\sigma\Big(1-\frac{q-1}{q}\frac{|y|^2-2n}{4|\log\sigma|}
		+\frac{\mathcal{R}(y,|\log \sigma|)}{|\log\sigma|}\Big)^{1/(1-q)}\Big]\quad\hbox{if }\ q>1;
		\end{cases}
		$$
		with
		$$\lim_{s\to\infty}\|\mathcal{R}(\cdot,s)\|_{H^1_\rho}=0\quad\hbox{ and uniformly for bounded $y$}.$$
		
	$\bullet$ {\bf Case $q>1$.} 
 Recall that $f$ is regularly varying at $\infty$ with index $q/(q-1)$; see Remark~\ref{CSLem}. In particular, $G^{-1}$ is regularly varying at $0$ with index $1-q$. Hence, for bounded $|y|$ and sufficiently small $\sigma$, we obtain
	$$
	u(x,T-\sigma)\ge G^{-1}(\sigma)\left(1+ \frac{2n(q-1)}{4q|\log\sigma|}-\frac{(q-1)|x|^2}{4q\sigma|\log\sigma|}
		+\frac{\mathcal{R}_\sigma(x)}{|\log\sigma|}\right),
	$$		
		where 
		$$
		 \mathcal{R}_\sigma(x)=\mathcal{R}(x\sigma^{-1/2},|\log \sigma|)+\mathcal{R}_0(x\sigma^{-1/2},|\log \sigma|),\ \hbox{with }\ \|\mathcal{R}_0(\cdot,s)\|_{H_\rho^1}\underset{s\to \infty}{\to }0.
		$$ 
		
	$\bullet$ {\bf Case $q=1$.} Using \eqref{qlim} and Taylor expansion for bounded $|y|$ and sufficiently small $\sigma$, we obtain
	 	$$
	 	\frac{u(x,T-\sigma)}{G^{-1}(\sigma)}\ge 1+ \frac{f(G^{-1}(\sigma))}{G^{-1}(\sigma)f'(G^{-1}(\sigma))}\left(\frac{2n}{4|\log\sigma|}-\frac{|x|^2}{4\sigma|\log\sigma|}
		+\frac{\mathcal{R}_\sigma(x)}{|\log\sigma|}\right)
	 	$$
		where 
		$$
		 \mathcal{R}_\sigma(x)=\mathcal{R}(x\sigma^{-1/2},|\log \sigma|)+\mathcal{R}_1(x\sigma^{-1/2},|\log \sigma|),\ \hbox{with }\ \sup_{s\ge s_0}\|\mathcal{R}_1(\cdot,s)\|_{H_\rho^1}<\infty.
		$$ 

In all cases, we then have 
\be{usigma}
u(x,T-\sigma)\ge G^{-1}(\sigma)\left(1+ \frac{2ng(q,\sigma)}{4|\log\sigma|}-\frac{g(q,\sigma)|x|^2}{4\sigma|\log\sigma|}
		+\frac{\mathcal{R}_2(\sigma,x)}{|\log\sigma|}\right),
\ee
	where 
	\be{defgqsig}
	0\le g(q,\sigma)=\begin{cases} \frac{q-1}{q}\quad\hbox{if }\ q>1\\
	\frac{f(G^{-1}(\sigma))}{G^{-1}(\sigma)f'(G^{-1}(\sigma))} \quad\hbox{if }\ q=1;
	\end{cases}
	\ee
	and 
	$$
	 \mathcal{R}_2(\sigma,x)=\mathcal{R}(x\sigma^{-1/2},|\log \sigma|)+\mathcal{R}_3(x\sigma^{-1/2},|\log \sigma|),\quad \hbox{with }\quad \|\mathcal{R}_3(\cdot,s)\|_{H_\rho^1}\underset{s\to \infty}{\to }0.
	 	$$ 
Therefore we have 
$$
\begin{aligned}
 V(x,t)&= M-S(t-t_0) M+S(t-t_0)\Big( u(t_0)\Big)\ge 
S(t-t_0)\Big( u(t_0\Big)\\
&\ge G^{-1}(\sigma)S(t-t_0)\Bigl(1+\frac{2ng(q,\sigma)}{4|\log\sigma|}-\frac{g(q,\sigma)|x|^2}{4\sigma|\log\sigma|}\Bigl)
+\frac{G^{-1}(\sigma)}{|\log\sigma|^2}S(t-t_0)\mathcal{R}_2\equiv V_1+V_2\\
\end{aligned}$$
hence, since $ V\ge 0$,
\be{VV1V2}
V\ge (V_1)_+-|V_2|\quad \hbox{ in $B_\delta\times(t_0,T)$.}
\ee
Using Lemma~\ref{lemSrho}(i) and recalling \eqref{deft0}, 
we may estimate $(V_1)_+$ from below for $t\in[t_0,T)$ as:
\be{V1below}
\begin{aligned}
(V_1(t))_+
&\ge e^{-\lambda(t-t_0)} G^{-1}(\sigma)\Bigl[1+\frac{2ng(q,\sigma)}{4|\log\sigma|}-\frac{g(q,\sigma)(|x|^2+2n(t-t_0))}{4\sigma|\log\sigma|}\Bigl]_+
\chi_{B_{\delta/2}} \\
&{\strut =\strut} e^{-\lambda\sigma} G^{-1}(\sigma) \Bigl[1-\frac{g(q,\sigma)|x|^2}{4\sigma|\log\sigma|}\Bigl]_+
\chi_{B_{\delta/2}}. 
\end{aligned}
\ee
To control $V_2$, we observe that, for $t\in[t_0,T)$,
$$
\begin{aligned}
\int_{B_\delta}\mathcal{R}_2^2(\sigma,z)e^{-\frac{|z|^2}{4(t-t_0)}}dz&
\le 2\int_{B_\delta} (\mathcal{R}^2+\mathcal{R}_3^2)(z\sigma^{-1/2},|\log \sigma|)e^{-\frac{|z|^2}{4\sigma}}dz\\
&\le 2\sigma^{n/2}(\|\mathcal{R}(\cdot,|\log \sigma|)\|^2_{L^2_\rho}+\|\mathcal{R}_3(\cdot,|\log \sigma|)\|^2_{L^2_\rho})
\end{aligned}
$$
and then we use Lemma~\ref{lemSrho}(ii) 
to write, for $x\in B_\delta,\ t\in[t_0,T)$:
$$\begin{aligned}
\bigl|S(t-t_0)\mathcal{R}_2(\sigma,\cdot)\bigr|(x)
&\le C(n)(t-t_0)^{-n/4}\bigg(\int_{B_\delta}(\mathcal{R}_2^2(\sigma,z)e^{-\frac{|z|^2}{4(t-t_0)}}dz\bigg)^{1/2}e^{\frac{|x|^2}{2(t-t_0)}}\\
&\le C(n)(t-t_0)^{-n/4}\sigma^{n/4}e^{\frac{|x|^2}{2(t-t_0)}}(\|\mathcal{R}(\cdot,|\log \sigma|)\|_{L^2_\rho}+\|\mathcal{R}_3(\cdot,|\log \sigma|)\|_{L^2_\rho}).
\end{aligned}$$
Denoting $D_\sigma=\overline B_{\sqrt\sigma}\times [T-\textstyle\frac{\sigma}{2},T)$ and using \eqref{deft0},
we obtain
\be{controlV2}
\sup_{D_\sigma}|V_2|
\le C(n)G^{-1}(\sigma)\frac{\|\mathcal{R}(\cdot,|\log \sigma|)\|_{L^2_\rho}+\|\mathcal{R}_3(\cdot,|\log \sigma|)\|_{L^2_\rho}}{|\log\sigma|}
=\frac{\eps(\sigma)G^{-1}(\sigma)}{|\log\sigma|}.
\ee
Here and in the rest of the proof $\eps$ denotes a generic function such that $\lim_{\sigma\to 0}\eps(\sigma)=0$.
 Assuming $\sigma\le \delta^2/4$, it follows from \eqref{VV1V2}-\eqref{controlV2} that 
\be{controlV}
\tilde V\ge e^{-\lambda\sigma} G^{-1}(\sigma)\Bigl(1-\frac{g(q,\sigma)+\eps(\sigma)}{4|\log\sigma|}\Bigl)\quad\hbox{in $D_\sigma$}.
\ee
Moreover, we may choose $\sigma_0\in(0,\delta^2/4)$ such that the RHS of \eqref{controlV} is $> M$ 
for all $\sigma\in(0,\sigma_0]$.
Therefore
$$
0<G(V)+t_0-t\le  \mu_q(\sigma,t):=G\biggl[e^{-\lambda\sigma} G^{-1}(\sigma)\Bigl(1-\frac{g(q,\sigma)+\eps(\sigma)}{4|\log\sigma|^2}\Bigl)\biggr]+T-t-\sigma
\quad\hbox{in $D_\sigma$},
$$
where the RHS is $<G( M)$ hence, by \eqref{a1}-\eqref{defUV},
\be{fact}
u\ge G^{-1}(\mu_q(\sigma,t))\quad\hbox{in $D_\sigma$}.
\ee

\vspace{.2cm}

\noindent{\bf Step 3.} {\it Asympotics of $\mu_q(\sigma,t)$ and conclusion.}
Fix any $x_0$ such that $|x_0|^2\le\sigma_0$. Choosing 
\be{choicesigma}
\sigma=|x_0|^2,
\ee
it follows from \eqref{fact} that
\be{fact2}
u(x_0, t)\ge G^{-1}(\mu_q(\sigma,t))\quad\hbox{ for $T-\frac{\sigma}{2}\le t<T$.}
\ee

To estimate $\mu_1$ from above, we first use \eqref{ell12a2} in Lemma~\ref{ell12a} to get, for some constant $C>0$,
$$
\mu_1(\sigma,t)\le (1+C\sigma)G\biggl[G^{-1}(\sigma)\Bigl(1-\frac{g(q,\sigma)+\eps(\sigma)}{4|\log\sigma|}\Bigl)\biggr]+T-t-\sigma.
$$
Next, by $G'=-1/f$ and the mean value theorem, there exists $\theta(\sigma)\in(0,1)$ such that

$$
\mu_1(\sigma,t)\le (1+C\sigma)\biggl\{\sigma+
\frac{
\frac{g(q,\sigma)+\eps(\sigma)}{4|\log\sigma|}G^{-1}(\sigma)}{f\Bigl[G^{-1}(\sigma)\Bigl(1-\theta(\sigma)\frac{g(q,\sigma)+\eps(\sigma)}{4|\log\sigma|}\Bigl)\Bigl]}\biggr\}+T-t-\sigma.
$$

By \eqref{ell12a1} in Lemma~\ref{ell12a} and the fact that $g(q,\sigma)\le C$ for any $q\ge 1$ and small $\sigma$, we have
$$
f\bigl[G^{-1}(\sigma)\bigl(1-\theta(\sigma)\frac{g(q,\sigma)+\eps(\sigma)}{4|\log\sigma|}\bigl)\bigl]
\ge (1+\eps(\sigma))f\bigl(G^{-1}(\sigma)).
$$

\begin{itemize}

\item {\bf \underline{$q=1$ case}:}
Using the definition of $g(q,\sigma)$ in \eqref{defgqsig} for $q=1$, it follows that
$$
\begin{aligned}
\mu_1(\sigma,t)
&\le (1+C\sigma)\biggl\{\sigma+\frac{(g(q,\sigma)+\eps(\sigma))G^{-1}(\sigma)}{4|\log\sigma|f(G^{-1}(\sigma))}\biggr\}+T-t-\sigma\\
&=T-t+C\sigma^2+ \frac{(g(q,\sigma)+\eps(\sigma))G^{-1}(\sigma)}{4|\log\sigma|f(G^{-1}(\sigma))}\\
&\le T-t+C\sigma^2+ \frac{(1+\eps(\sigma))}{4|\log\sigma|}\times\frac{g(q,\sigma)G^{-1}(\sigma)}{f(G^{-1}(\sigma))}\\
&\le T-t+C\sigma^2+\frac{(1+\eps(\sigma))\sigma}{4|\log\sigma|(f'G)(G^{-1}(\sigma))}\\
&\le(1+\eps(\sigma))\Bigl(T-t+\frac{\sigma}{4|\log\sigma|}\Bigr).
\end{aligned}$$

\item {\bf \underline{$q>1$ case}:}
Using $g(q,\sigma)=(q-1)/q$ along with the fact that $X/(G(X)f(X))\to 1/(q-1)$ as $X\to \infty$, we obtain
$$
\begin{aligned}
\mu(\sigma,t)
&\le (1+C\sigma)\biggl\{\sigma+\frac{(g(q,\sigma)+\eps(\sigma))G^{-1}(\sigma)}{4|\log\sigma|f(G^{-1}(\sigma))}\biggr\}+T-t-\sigma\\
&=T-t+C\sigma^2+ \frac{(q-1+\eps(\sigma))G^{-1}(\sigma)}{4q|\log\sigma|f(G^{-1}(\sigma))}\\
&\le T-t+C\sigma^2+ \frac{(1+\eps(\sigma))\sigma}{4q|\log\sigma|}\times\frac{(q-1)G^{-1}(\sigma)}{(Gf)(G^{-1}(\sigma))}\\
&\le T-t+C\sigma^2+\frac{(1+\eps(\sigma))\sigma}{4q|\log\sigma|}\\
&\le(1+\eps(\sigma))\Bigl(T-t+\frac{\sigma}{4q|\log\sigma|}\Bigr).
\end{aligned}$$
\end{itemize}

By \eqref{choicesigma}, \eqref{fact2}, using \eqref{equivGG} again, we thus deduce that for all $q\ge 1$, as $(x,t)\to (0,T)$,
$$u(x,t)\ge(1+o(1))G^{-1}\bigg(T-t+\frac{|x|^2}{4q|\log|x|^2|}\bigg),
\quad\hbox{ uniformly for $ |x|^2\ge 2(T-t)$.}$$
		
	On the other hand, by the second assertion of Proposition~\ref{qssblowup} and \eqref{equivGG} again, 
	 we have, as $(x,t)\to (0,T)$,
		$$u(x,t)= (1+o(1))G^{-1}(T-t)=(1+o(1))G^{-1}\bigg(T-t+\frac{|x|^2}{4|\log|x|^2|}\bigg),
		\ \hbox{ uniformly for $ |x|^2\le 2(T-t)$.}$$
		 Estimate \eqref{infer1-lower}, and so the lower part of \eqref{infer1}, follows.
 \end{proof}

\noindent{\bf Acknowlegement.} The author thanks Prof.~Philippe Souplet for useful suggestions during the preparation of this work.

\smallskip

\noindent{\bf Statements and Declarations.} The author states that there is no conflict of interest. 
This manuscript has no associated data.

\end{document}